\documentclass[12pt]{article}
\usepackage{a4}
\usepackage{latexsym}
\usepackage{amsfonts}
\usepackage{amssymb}
\usepackage{amsmath}
\usepackage{amsthm}
\newtheorem{theorem}{Theorem}

\newtheorem{lemma}[theorem]{Lemma}
\newtheorem{corollary}[theorem]{Corollary}
\newtheorem{proposition}[theorem]{Proposition}
\newtheorem{definition}[theorem]{Definition}

\title{Improved bounds on the set $A(A+1)$}
\author{Timothy G. F. Jones\footnote{School of Mathematics, University of Bristol, University Walk,
 Bristol, BS8 1TW, United Kingdom. Email: tgf.jones@bristol.ac.uk} and Oliver Roche-Newton\footnote{School of Mathematics, University of Bristol, University Walk,
 Bristol, BS8 1TW, United Kingdom. Email: maorn@bristol.ac.uk}}
\date{}

\begin{document}
\maketitle

\begin{abstract}
For a subset $A$ of a field $F$, write $A(A+1)$ for the set $\left\{a(b+1):a,b \in A\right\}$. We establish new estimates on the size of $A(A+1)$ in the case where $F$ is either a finite field of prime order, or the real line. 

In the finite field case we show that $A(A+1)$ is of cardinality at least $C|A|^{57/56-o(1)}$ for some absolute constant $C$, so long as $|A|<p^{1/2}$. In the real case we show that the cardinality is at least $C|A|^{24/19-o(1)}$. These improve on the previously best-known exponents of $106/105-o(1)$ and $5/4$ respectively. 
\end{abstract}


\section{Introduction}

Throughout this paper we use  $Y=O(X)$, $X=\Omega(Y)$ and $Y \ll X$ all to mean that there is an absolute constant $C$ with $Y\leq CX$. We also use $Y \lesssim X$ to mean that there is an absolute constant $c$ with $Y\ll \log(X)^{c} X$. 

The well-known sum-product phenomenon says that for a set $A \subseteq \mathbb{R}$, at least one of the sum set $A+A=\left\{a+b:a,b \in A\right\}$ and the product set $AA=\left\{ab:a,b \in A\right\}$ will be large. Here, `large' means  having cardinality $\Omega\left(|A|^{1+\delta}\right)$ for some absolute $\delta>0$. The best-known exponent when $A$ is a subset of $\mathbb{R}$ is $\delta=1/3-o(1)$, due to Solymosi \cite{solymosi1}, improving on previous work in \cite{elekes,erdosszemeredi,ford,nathanson,solymosi2}. The best-known exponent when $A$ is a subset of $\mathbb{F}_p$ is $\delta=1/11-o(1)$, due to Rudnev \cite{rudnev}, improving on \cite{BG,BKT,garaev,ks,li,shen}. This bound was extended to the general finite field $\mathbb{F}_q$ by by Li and the second-named author \cite{LiORN1}, following previous work of Katz and Shen \cite{KSgeneral}.

A variation on this theme is to find functions $f$ of two variables with the property that for any set $A$, the set $f(A)=\left\{f(a,b):a,b \in A\right\}$ has cardinality $\Omega\left(|A|^{1+\delta}\right)$. The study of such functions was initiated by Bourgain \cite{bourgain} in a finite field setting. The strongest results are due to Garaev and Shen \cite{GS}, who proved three results about the size of the set $A(A+1)$, depending on the ambient field and the density of $A$ within it:

\begin{align}
&\label{eq:gs1}\text{If } A \subseteq \mathbb{F}_p \text{ with } |A|<p^{1/2} \text{ then } |A(A+1)|\gg |A|^{106/105}\\
&\label{eq:gs2}\text{If } A \subseteq \mathbb{F}_p \text{ with } |A|\geq p^{2/3} \text{ then }|A(A+1)|\gg |A|^{1/2}p^{1/2}\\
&\label{eq:gs3}\text{If } A \subseteq \mathbb{R} \text { is finite then } |A(A+1)|\gg |A|^{5/4}.
\end{align}

Result (\ref{eq:gs2}) is sharp but (\ref{eq:gs1}) and (\ref{eq:gs3}) are not. This paper makes the following improvements on these cases:

\begin{theorem}\label{theorem:mainfinite}
For $A \subseteq \mathbb{F}_p$ with $|A|<p^{1/2}$ we have $\left|A(A+1)\right|\gtrsim |A|^{57/56}$.
\end{theorem}

\begin{theorem}\label{theorem:mainreal} For any finite set $A\subset{\mathbb{R}}$ we have $|A(A+1)|\gtrsim |A|^{24/19}.$
\end{theorem}

In what follows, Section \ref{section:finite} gives the proof of Theorem \ref{theorem:mainfinite} and Section \ref{section:real} gives the proof of Theorem \ref{theorem:mainreal}. 

\section{The finite field case}\label{section:finite}

The overall strategy for proving Theorem \ref{theorem:mainfinite} is similar to the finite-field case in \cite{GS}. First, we bound a partial sumset from above in terms of $A(A+1)$. Second, we use the sum-product-type fact that if the partial sumset is small then $A(A+1)$ must be large. In combination, these show that $A(A+1)$ must always be large.

Our new bound follows from innovations in both parts of the strategy. In the first part, we find a more efficient upper bound on partial sumsets, via the Ruzsa-type observation that if $ab=cd$ then $(a-ab)-(c-cd)=a-c$. For the second part, we incorporate partial sumsets more efficiently into the sum-product framework. 

It is worth remarking that deploying the most recent innovation in sum-product techniques from \cite{rudnev} should lead to a further slight improvement in the result - we estimate improving the exponent from $57/56$ to $52/51$. But we leave this for another time, since the purpose of this paper is to introduce some new technical ideas, and the 52/51 bound would introduce even more technicality and make the proof less accessible. 

Section \ref{section:partial} establishes a new upper bound on partial sumsets. Section \ref{section:lemmata} records some triangle and covering lemmata from sumset calculus that we will need. Section \ref{section:proof} then uses the results of the previous two sections to prove Theorem \ref{theorem:mainfinite}.
 
\subsection{Bounding partial sumsets}\label{section:partial}

We first recall the definition of a partial sumset.

\begin{definition}
Let $A,B$ be sets and $G \subseteq A \times B$. We define the \textbf{partial sumset} $A\overset{G}+B=\left\{a+b:(a,b)\in G\right\}$. 
\end{definition}

This section proves the following upper bound for certain partial sumsets.

\begin{proposition}\label{theorem:idea}
Let $A,B\subseteq \mathbb{F}_p$, and let $\epsilon>0$. There exists $G \subseteq A \times B$ with $|G|\geq (1-\epsilon)|A||B|$ such that 
$$\left|A\stackrel{G}{-} B\right|\ll_{\epsilon} \frac{|A(B+1)||B(A+1)||A/B|}{|A||B|}.$$
\end{proposition}

\begin{proof}
Without loss of generality assume that $0 \notin A,B$. Note that $\sum_{x \in A/B}\left|A \cap xB\right|=|A||B|.$ Let $X$ be the set of $x \in A/B$ for which $|A \cap xB|\geq \frac{\epsilon|A||B|}{|A/B|}$. Then

\begin{align*}
|A||B| &= \sum_{x \in X}|A \cap xB|+\sum_{x \notin X}|A \cap xB|\\
& \leq \sum_{x \in X}|A \cap xB| + \epsilon|A||B|
\end{align*}

and so $\sum_{x \in X}|A \cap xB| \geq (1-\epsilon)|A||B|$. Let $G=\left\{(a,b) \in A \times B:\frac{a}{b} \in X\right\}$ so that 

$$|G|=\sum_{x \in X}|A\cap xB|\geq (1-\epsilon) |A||B|.$$

For each $\xi \in A \stackrel{G}{-}B$ pick $a(\xi) \in A ,b(\xi)\in B$ such that $a(\xi)-b(\xi)=\xi$. Let

$$S=\left\{\left(\xi,(c,d)\right):\xi \in A\stackrel{G}{-}B,(c,d)\in A \times B, \frac{c}{d}=\frac{a(\xi)}{b(\xi)}\right\}$$

and note that

$$|S|\gg_{\epsilon} \frac{|A||B|}{|A/B|}\left|A\stackrel{G}{-}B\right|$$

since every element of $A\stackrel{G}{-}B$ occurs once as an abscissa of the set $S$, and each such abscissa has $\Omega\left(\frac{|A||B|}{|A/B|}\right)$ associated ordinates. We now show that $|S| \leq \left|A(B+1)\right|\left|B(A+1)\right|$. This will follow once we have shown that the map 

$$f:S \to A(B+1) \times B(A+1)$$ 
$$f(\xi,(c,d))= \left(a(\xi)+a(\xi)d,b(\xi)+b(\xi)c\right)$$

is an injection. Indeed for given $(t_1,t_2)$ in $f(S)$ we know 

$$\xi = a(\xi)-b(\xi)=\left(a(\xi)+a(\xi)d\right)-\left(b(\xi)+b(\xi)c\right)=t_1-t_2 $$

so we know $\xi$ and therefore $a(\xi)$ and $b(\xi)$. We therefore also know $(c,d)$ since $t_1=a(\xi)+a(\xi)d$ and $t_2=b(\xi)+b(\xi)c$. There is therefore only one choice of $(\xi,(c,d))\in S$ that can map to $(t_1,t_2)$.

Combining all of the above we see that 

$$\frac{|A||B|}{|A/B|}\left|A\stackrel{G}{-}B\right| \ll_{\epsilon} \left|A(B+1)\right|\left|B(A+1)\right|$$

and so $\left|A\stackrel{G}{-}B\right|\ll_{\epsilon} \frac{|A(B+1)||B(A+1)||A/B|}{|A||B|}$ as required.
\end{proof}

\subsection{Triangle and covering results}\label{section:lemmata}

We collect here a covering result and some variants of the Pl\"unnecke-Ruzsa triangle inequalities from sumset calculus. Note that whilst these are recorded in their additive form, they can also be applied without further remark to product sets.  We begin with the following lemma, which is useful when analysing particularly dense partial sumsets. Note that by the $G$-degree of $a \in A$ we mean the number of $b \in B$ such that $(a,b)\in G$. 

\begin{lemma}\label{theorem:denseG}Let $G \subseteq A \times B$ with $|G|\geq (1-\epsilon)|A||B|$. There exists $A' \subseteq A$ with $|A'|\geq (1-\sqrt{\epsilon})|A|$ such that the $G$-degree of every $a \in A'$ is at least $(1-\sqrt{\epsilon})|B|$. 
\end{lemma}

\begin{proof}
Let $A'\subseteq A$ be the set of $a \in A$ with $G$-degree at least $(1-\sqrt{\epsilon})|B|$. We will show that $|A'|\geq (1-\sqrt{\epsilon})|A|$. For each $a \in A$ let $B_a$ be the set of $b \in B$ such that $(a,b)\in G$. We have

\begin{align*}
(1-\epsilon)|A||B|& \leq |G|\\
&= \sum_{a \in A}|B_a|\\
&\leq \sum_{a \in A'}|B| + \sum_{a \in A \setminus A'}\left(1-\sqrt{\epsilon}\right)|B|\\
&\leq |A'||B| + (|A|-|A'|)\left(1-\sqrt{\epsilon}\right)|B|.
\end{align*}

We therefore obtain

\begin{align*}
(1-\epsilon)|A| &\leq |A'|+(|A|-|A'|)\left(1-\sqrt{\epsilon}\right)\\
&= (1-\sqrt{\epsilon})|A|+ \sqrt{\epsilon}|A'|
\end{align*}

and so $|A'|\geq (1-\sqrt{\epsilon})|A|$ as required. 
\end{proof}

The following covering result restates and slightly generalises one due to Shen \cite{shen}, showing that if $A \overset{G}-B$ is small for a large but not necessarily complete $G$, then much of $A$ can be covered by few translates of $B$.

\begin{lemma}\label{theorem:cover}
Let $G \subseteq A \times B$ and $0<\epsilon<1/4$. If $|G|\geq (1-\epsilon)|A||B|$ then there exists $A' \subseteq A$ with $|A'|\geq (1-2\sqrt{\epsilon}) |A|$ such that $A'$ is contained in the union of $O_{\epsilon}\left(\frac{|A\overset{G}-B|}{|B|}\right)$ translates of $B$. Similarly, we can find a subset $A'' \subseteq A$ with $|A''|\geq (1-2\sqrt{\epsilon}) |A|$  such that $A''$ in contained in the union of $O_{\epsilon}\left(\frac{|A\overset{G}-B|}{|B|}\right)$ translates of $-B$.
\end{lemma}

\begin{proof}
We shall prove the case for covering with translates of $B$, and remark on the slight alteration needed to cover with translates of $-B$. Since $|G|\geq (1-\epsilon)|A||B|$ there is by Lemma \ref{theorem:denseG} a subset $A_1 \subseteq A$ with $|A_1|\geq (1-\sqrt{\epsilon}) |A|$ such that every element of $A_1$ has $G$-degree at least $(1-\sqrt{\epsilon}) |B|$. 

Now for any subset $A_* \subseteq A_1$ let $G_*=G \cap (A_* \times B)$ so that $|G_*| \geq (1-\sqrt{\epsilon})|A_*||B|$. Write $E_+(A_*,B)$ for the additive energy of $A_*$ and $B$, i.e. the number of solutions to $a+b=a'+b'$ with $a,a' \in A_*$ and $b,b' \in B$. This is the same as the number of solutions to $a-b=a'-b'$, which is in turn at least the number of solutions for which $(a,b)$ and $(a',b')$ are both in $G^*$.

By Cauchy-Schwarz and the lower bound on $|G_*|$ we therefore have

$$E_+(A_*,B) \geq \frac{|G_*|^2}{|A\overset{G}-B|}\geq \frac{|A_*|^2(1-\sqrt{\epsilon})^2|B|^2}{|A\overset{G}-B|}.$$

Since 

\begin{equation}\label{eq:energydef}
E_+(A^*,B)=\sum_{a \in A^*, b \in B}|A_* \cap (a-b)+B |
\end{equation}

this means that for any $A_* \subseteq A_1$ there exists $a \in A_*$, $b \in B$ such that 

$$|A_* \cap (a-b)+B |\geq \frac{|A_*| (1-\sqrt{\epsilon})^2 |B|}{|A\overset{G}-B|}.$$

We apply the above discussion to a sequence of subsets of $A$. We begin by taking $A_*=A_1$ to find $a_1 \in A_1$, $b_1 \in B$ such that

$$|A_1 \cap (a_1-b_1)+B |\geq \frac{|A_1| (1-\sqrt{\epsilon})^2 |B|}{|A\overset{G}-B|}.$$

So the translate $(a_1-b_1)+B$ covers $\frac{|A_1| (1-\sqrt{\epsilon})^2 |B|}{|A\overset{G}-B|}$ elements of $A_1$. We  discard $A_1 \cap (a_1-b_1)+B$ from $A_1$ and let $A_2$ be the set of elements remaining, now taking $A_*=A_2$ and repeating the process. We iterate $O_{\epsilon}\left(\frac{|A\overset{G}-B|}{|B|}\right)$ times until the set remaining is of cardinality no more than $\sqrt{\epsilon}|A_1|$. We then take $A'$ to be the set of elements discarded across all iterations, so that $|A'|\geq (1-\sqrt{\epsilon})|A_1|$. Since $|A_1|\geq (1-\sqrt{\epsilon})|A|$ we get $|A'|\geq (1-2\sqrt{\epsilon})|A|$ as required, which completes the proof for covering with translates of $B$.

The proof for covering with translates of $-B$ is identical, except that in place of (\ref{eq:energydef}) we use $E_+(A^*,B)=\sum_{a \in A^*, b \in B}|A_* \cap (a+b)-B |$.
\end{proof}

Our next result is the well-known Ruzsa triangle inequality, see e.g. \cite{TV} for a proof.

\begin{lemma}\label{theorem:ruzsa}
For sets $A,B,C$ we have
$$|A-B|\leq \frac{|A-C||B-C|}{|C|}.$$
\end{lemma}

We now prove an extension to Lemma \ref{theorem:ruzsa} to high-density partial sumsets. This is set as Exercise 2.5.4 in the book of Tao and Vu \cite{TV}. In the case $\epsilon=0$ it collapses to the statement of the Lemma \ref{theorem:ruzsa}.

\begin{lemma}\label{theorem:semibsg}
Let $0<\epsilon<1/4$ and let $G \subseteq A \times B$ and $H \subseteq B \times C$, such that $|G|\geq (1-\epsilon)|A||B|$ and $|H|\geq (1-\epsilon)|B||C|$. Then there exist $A' \subseteq A$ and $C' \subseteq C$ with $|A'|\geq (1-\sqrt{\epsilon})|A|$ and $|C'|\geq (1-\sqrt{\epsilon})|C|$ such that 

$$|A'-C'|\ll_{\epsilon}\frac{|A\overset{G}-B||B\overset{H}-C|}{|B|}.$$
\end{lemma}

\begin{proof}
Since $|G|\geq(1-\epsilon)|A||B|$ there exists by Lemma \ref{theorem:denseG} a set $A' \subseteq A$ with $|A'|\geq (1-\sqrt{\epsilon})|A|$ such that every $a \in A'$ has $G$-degree at least $(1-\sqrt{\epsilon})|B|$. Similarly, there exists $C' \subseteq C$ with $|C'|\geq (1-\sqrt{\epsilon})|C|$ such that every $c \in C'$ has $H$-degree at least $(1-\sqrt{\epsilon})|B|$. We shall show that $|A'-C'|$ satisfies the required upper bound.

For each $a \in A'$ let $B_a$ be the set of $b \in B$ such that $(a,b)\in G$. And for each $c \in C'$ let $B_c$ be the set of $b \in B$ such that $(b,c)\in H$. For each $a \in A', c \in C'$ we know that $|B_a|,|B_c|\geq (1-\sqrt{\epsilon})|B|$ and so $|B_a \cap B_c|\geq (1-2\sqrt{\epsilon})|B|$.

For each $x \in A'-C'$, fix $a(x) \in A', c(x) \in C'$ such that $a(x)-c(x)=x$. Let 

$$Y=\left\{(x,b) : x\in A'-C', b \in B_{a(x)}\cap B_{c(x)}\right\}.$$

We know that $|Y| \geq (1-2\sqrt{\epsilon})|B||A'-C'|$. And the injection 

$$(x,b)\mapsto (a(x)-b,b-c(x))$$ 

from $Y$ to $A\overset{G}-B \times B \overset{H}-C$ shows that $|Y|\leq |A\overset{G}-B||B\overset{H}-C|$. Comparing the upper and lower bounds on $Y$, we are done. 

\end{proof}

We record a consequence that will be of use to us.

\begin{corollary}\label{theorem:semibsgapp}
Let $0<\epsilon<1/16$ and let $G \subseteq A \times A$ with $|G| \geq (1-\epsilon)|A|^2$. Then there exists $A' \subseteq A$ with $|A'|\gg_{\epsilon}|A|$ such that 

$$|A'-A'|\ll_{\epsilon}\frac{|A \overset{G}-A|^2}{|A|}.$$
\end{corollary}

\begin{proof}
Apply Lemma \ref{theorem:semibsg} with $A=B=C$ and $H=G$. There exist $A_1,A_2 \subseteq A$ with $|A_1|,|A_2|\geq (1-2\sqrt{\epsilon})|A|$ such that $|A_1-A_2|\ll \frac{|A\overset{G}-A|^2}{|A|}$. We then take $A'=A_1 \cap A_2$. Since $\epsilon<1/16$ we know that $1-2\sqrt{\epsilon}>1/2$ and so $|A'|\gg_{\epsilon}|A|$. Moreover, $|A'-A'|\leq |A_1-A_2| \leq \frac{|A\overset{G}-A|^2}{|A|}$.
\end{proof}

Our final triangle result is the Pl\"unnecke-Ruzsa inequality. We use a variant due to Katz and Shen \cite{ks}.

\begin{lemma}\label{theorem:pl}
For sets $A,X_1,\ldots,X_k$, there exists $A' \subseteq A$ with $|A'| \approx |A|$ such that 
$$\left|A'+X_1+\ldots X_k\right|\ll \frac{\prod\limits_{j=1}^k \left|A+X_j\right|}{|A|^{k-1}}.$$
\end{lemma}

\subsection{Proving Theorem \ref{theorem:mainfinite}}\label{section:proof}

As a preliminary step we record two consequences of applying the results of Section \ref{section:lemmata} to the partial sumset bound from Proposition \ref{theorem:idea}. Our first consequence is the least-efficient, due to its use of Corollary \ref{theorem:semibsgapp}. Because of this we will employ it only when compelled.

\begin{corollary}\label{theorem:horrific}
For any set $A \subseteq \mathbb{F}_p$ there exists $A'\subseteq A$ with $|A'|\gg |A|$ such that $\left|A'-A'\right|\ll \frac{|A(A+1)|^8}{|A|^7}.$
\end{corollary}

\begin{proof}
Apply Proposition \ref{theorem:idea} with $A=B$ to find $G \subseteq A \times A$ with $|G|\approx |A|^2$ such that $|A \overset{G}-A|\ll \frac{|A(A+1)|^2|A/A|}{|A|^2}.$ By Corollary \ref{theorem:semibsgapp} there exists $A' \subseteq A$ with $|A'|\approx |A|$ such that $|A'-A'|\ll \frac{|A(A+1)|^4|A/A|^2}{|A|^5}.$ By Lemma \ref{theorem:ruzsa} applied multiplicatively we have $|A/A|\leq \frac{|A(A+1)|^2}{|A|}$ and so obtain the result.
\end{proof}

Our second consequence follows by applying Lemma \ref{theorem:cover} to Proposition \ref{theorem:idea}

\begin{corollary}\label{theorem:energy}
Let $A,B,C \subseteq \mathbb{F}_p$, and $A,B\subseteq xC+y$ for some $x \in \mathbb{F}_p^*, y \in \mathbb{F}_p$. Let $0<\epsilon<1/16$. Then $(1-\epsilon)|A|$ elements of $A$ can be covered by $$O_{\epsilon}\left( \frac{|C(C+1)|^2|C/C|}{|A||B|^2}\right)$$ translates of $B$. Similarly, $(1-\epsilon)|A|$ elements of $A$ can be covered by this many translates of $-B$. 
\end{corollary}

\begin{proof}
Applying Proposition \ref{theorem:idea} to the sets $A_{xy}=\frac{A-y}{x} \subseteq C$ and $B_{xy}=\frac{B-y}{x} \subseteq C$ we find $G_{xy}\subseteq A_{xy} \times B_{xy}$ of cardinality at least $(1-\epsilon^2/4)|A||B|$ such that 

\begin{align*}|A_{xy}\overset{G_{xy}}-B_{xy}|&\ll \frac{|A_{xy}(B_{xy}+1)||B_{xy}(A_{xy}+1)||A_{xy}/B_{xy}|}{|A||B|}\\
&\leq \frac{|C(C+1)|^2|C/C|}{|A||B|}.
\end{align*}

We then let $G=\left\{(a,b):\left(\frac{a-y}{x},\frac{b-y}{x}\right)\in G_{xy}\right\}$ to obtain $|A\overset{G}-B|\ll \frac{|C(C+1)|^2|C/C|}{|A||B|}.$ By applying Lemma \ref{theorem:cover} to $G$, the result follows.
\end{proof}

We now begin the proof of Theorem \ref{theorem:mainfinite}. By Corollary \ref{theorem:horrific} and passing to a subset if necessary we may assume that

\begin{equation}\label{eq:finite1}
|A-A|\ll \frac{|A(A+1)|^8}{|A|^7}.
\end{equation}

By Lemma \ref{theorem:pl} we may assume that

\begin{equation}\label{eq:finite2}
|A-A-A-A|\ll \frac{|A-A|^3}{|A|^2}.
\end{equation}

Now, by Cauchy-Schwarz we know that 

$$\sum_{a,b \in A}\left|a(A+1) \cap b(A+1) \right|\geq \frac{|A|^4}{|A(A+1)|}.$$

So there exists $b_0 \in A$ such that

$$\sum_{a \in A}\left|a(A+1) \cap b_0 (A+1)\right|\geq \frac{|A|^3}{|A(A+1)|}.$$

By dyadic pigeonholing we can find $A_1 \subseteq A$ and $N \in \mathbb{N}$ such that

$$\left|a(A+1) \cap b_0 (A+1)\right|\approx N$$ 

for all $a \in A_1$ and 

\begin{equation}\label{eq:finite3}
N|A_1|\gtrsim \frac{|A|^3}{|A(A+1)|}.
\end{equation}

Since $|A_1|\leq |A|$ we also have 

\begin{equation}\label{eq:finite4}
N\gtrsim \frac{|A|^2}{|A(A+1)|}.
\end{equation}

We now consider the set $$R(A_1)=\left\{\frac{\alpha-\beta}{\gamma-\delta}:\alpha,\beta,\gamma, \delta \in A_1\right\}$$ and break into two cases, according to whether or not $R(A_1)=\mathbb{F}_p$.

\subsubsection{$R(A_1)\neq \mathbb{F}_p$}
Since $R(A_1)\neq \mathbb{F}_p$ we can find $\xi=\frac{\alpha-\beta}{\gamma-\delta} \in R(A_1)$ such that $\xi-1 \notin R(A_1)$. Since $\xi-1 \notin R(A_1)$, the set $A_1+(\xi-1) A_1$ has no repetitions and so for any subset $A^* \subseteq A_1$ we have

\begin{align}
|A^*|^2 &\approx \left|A^* + A^*(\xi-1)\right|\nonumber\\
&\approx \left|A^*+ \frac{\alpha-\beta-\gamma + \delta}{\gamma-\delta}A^*\right|\nonumber\\
&=  \left|(\gamma - \delta)A^*+ (\alpha-\beta-\gamma + \delta)A^*\right|\nonumber\\
&\leq \left|(\gamma - \delta)A^*+ (\alpha-\beta)A^*-(\gamma - \delta)A^*\right|\label{eq:finite7}.
\end{align}

We now proceed to fix a particular choice of $A^*$. Let $\epsilon>0$ be sufficiently small. Applying Corollary \ref{theorem:energy} to the sets $\alpha(A_1+1)$ and $b_0(A_1+1)\cap \alpha(A_1+1)$ we see that there is a set $A_{\alpha} \subseteq A_1$ with $\left|A_{\alpha}\right|\geq (1-\epsilon)|A_1|$ such that $\alpha A_{\alpha}$ is contained in the union of $O_{\epsilon}\left(\frac{|A(A+1)|^2|A/A|}{N^2|A_1|}\right)$ translates of $b_0 A$. Similarly, we can find sets $A_{\beta},A_{\gamma},A_{\delta} \subseteq A_1$ such that $\beta A_{\beta}$ and $\gamma A_{\gamma}$ are contained in the union of 
$O_{\epsilon}\left(\frac{|A(A+1)|^2|A/A|}{N^2|A_1|}\right)$ translates of $b_0 A$, and $\delta A_{\delta}$ is contained in the union of $O_{\epsilon}\left(\frac{|A(A+1)|^2|A/A|}{N^2|A_1|}\right)$ translates of $-b_0 A$. We let  

$$A_2 = A_{\alpha}\cap A_{\beta}\cap A_{\gamma} \cap A_{\delta}$$

so that $|A_2|\geq (1-4\epsilon)|A_1|$. By Lemma \ref{theorem:pl} there is a set $A_3 \subseteq A_2$ with $|A_3|\gg |A_2|\gg |A_1|$ and the property that

$$\left|(\gamma - \delta)A_3+ (\alpha-\beta)A_3-(\gamma - \delta)A_3\right|\ll \frac{|A_2-A_2||(\alpha-\beta)A_2-(\gamma-\delta)A_2|}{|A_2|}.$$

We fix $A^*=A_3$ and so obtain from (\ref{eq:finite7}):

\begin{align*}
|A_1|^2 &\ll \left|(\gamma - \delta)A_3+ (\alpha-\beta)A_3-(\gamma - \delta)A_3\right|\\
&\ll \frac{|A_2-A_2||(\alpha-\beta)A_2-(\gamma-\delta)A_2|}{|A_2|}\\
&\leq \frac{|A-A||\alpha A_2- \beta A_2 - \gamma A_2 + \delta A_2|}{|A_1|}
\end{align*}

Since $A_2 \subseteq A_{\alpha},A_{\beta},A_{\gamma},A_{\delta}$, and $\alpha A_{\alpha}, \beta A_{\beta}, \gamma A_{\gamma}$
are each contained in the union of $O\left(\frac{|A(A_1)|^2|A/A|}{N^2|A_1|}\right)$ translates of $b_0 A$, and $\delta A_{\delta}$ is contained in the union of this many translates of $-b_0 A$, we have

\begin{align*}
|A_1|^2&\ll \frac{|A-A||\alpha A_{\alpha}- \beta A_{\beta} - \gamma A_{\gamma} + \delta A_{\delta}|}{|A_1|}\\
&\ll \frac{|A-A||b_0A-b_0A-b_0A-b_0A||A(A+1)|^8|A/A|^4}{N^8|A_1|^5}\\
&= \frac{|A-A||A-A-A-A||A(A+1)|^8|A/A|^4}{N^8|A_1|^5}.
\end{align*}

Then by (\ref{eq:finite2}) we have

\begin{equation}\label{eq:finite5}
|A_1|^2\ll \frac{|A-A|^4|A(A+1)|^8|A/A|^4}{N^8|A_1|^5|A|^2}.
\end{equation}

Now by Lemma \ref{theorem:ruzsa} applied multiplicatively we know

\begin{equation}\label{eq:finite6}
|A/A|\ll \frac{|A(A+1)|^2}{|A|}.
\end{equation}

Applying (\ref{eq:finite1}) and (\ref{eq:finite6}) to (\ref{eq:finite5}) we get

\begin{align*}
|A_1|^2&\ll \frac{|A(A+1)|^{48}}{N^8|A_1|^5|A|^{34}}.
\end{align*}

Rearranging and applying (\ref{eq:finite3}) and (\ref{eq:finite4}), we obtain 

$$|A(A+1)|^{48}\gg |A_1|^7 N^8 |A|^{34} \gtrsim \frac{N |A|^{55}}{|A(A+1)|^7} \gtrsim \frac{|A|^{57}}{|A(A+1)|^8}$$

and so $|A(A+1)|\gtrsim |A|^{57/56}$ as required.

\subsubsection{$R(A_1)= \mathbb{F}_p$}

Let $E$ be the number of solutions to 

\begin{equation}\label{eq:energyeq}
a+\xi b = c+\xi d 
\end{equation}

with $a,b,c,d \in A_1$ and $\xi \in R(A_1)$. Moreover, for each $\xi \in R(A_1)$ let $E(A_1,\xi A_1)$ denote additive energy of $A_1$ and $\xi A_1$, i.e. the number of solutions to (\ref{eq:energyeq}) with $\xi$ fixed, so that $E=\sum_{\xi \in R(A_1)}E(A_1,\xi A_1)$.

There are no more than $|R(A_1)||A_1|^2=p|A_1|^2$ solutions to (\ref{eq:energyeq}) for which $(a,b)=(c,d)$. And there are no more than $|A_1|^4$ solutions with $(a,b)\neq(c,d)$. So we have

$$\sum_{\xi \in R(A_1)}E(A_1,\xi A_1)\leq |A_1|^4 +p|A_1|^2.$$

Since $|A_1|\leq |A|<p^{1/2}$ this gives

$$\sum_{\xi \in R(A_1)}E(A_1,\xi A_1)\ll p|A_1|^2.$$
  
So there exists $\xi = \frac{\alpha-\beta}{\gamma-\delta} \in R(A_1)$ such that 

$$E(A_1,\xi A_1)\ll |A_1|^2.$$

We then know that for any $A^* \subseteq A_1$ with $|A_*|\approx |A_1|$ we have also

$$E(A^*,\xi A^*)\ll |A_1|^2.$$
 
Now by Cauchy-Schwarz we know that $E(A^*,\xi A^*)\geq \frac{|A^*|^4}{|A^*-\xi A^*|}\approx \frac{|A_1|^4}{|A^*-\xi A^*|} $ and so we obtain

\begin{align*}
|\alpha A^*- \beta A^* - \gamma A^* + \delta A^*|&\geq |A^*- \xi A^*|\gg |A_1|^2
\end{align*}

As before we fix $A_*=A_{\alpha}\cap A_{\beta} \cap A_{\gamma} \cap A_{\delta}$, and find ourselves in the same position as in the previous case, but with one less factor of $|A-A|$ to deal with. So we obtain (and in fact exceed) the required bound. \qed

\section{The real case}\label{section:real}

In order to prove Theorem \ref{theorem:mainreal}, we make use of recent innovations which have made use of bounds on additive energies of higher order. For sets $A,B$, we write $$E_{\alpha}(A,B)=\sum_{x\in A/B}\left|A \cap x B^{-1}\right|^{\alpha}.$$

Note that $E_2(A,B)$ is simply the multiplicative energy of $A$ and $B$, i.e. the number of solutions to the equation $ab=a'b'$ with $a,a'\in A$ and $b,b' \in B$. In the case $A=B$ we write $E_{\alpha}(A)=E_{\alpha}(A,A)$.

Schoen and Shkredov \cite{SS} made particular use of the case where $\alpha=3$ in order to establish a bound on the size of the sum set of a `convex' set, improving an earlier result of Elekes, Nathanson and Ruzsa \cite{ENR}. Subsequently, Li and the second-named author \cite{LiORN2} improved on some other results from \cite{ENR}, including a proof that

\begin{equation}
|A+f(A)|\gtrsim{|A|^{24/19}}
\label{convex}
\end{equation}

for any convex function $f$.

The upcoming proof follows a similar structure to the proof of (\ref{convex}), essentially improving on an application of the Szemer\'{e}di-Trotter by considering third order energy. In the same spirit, Rudnev \cite{mishaSP2} recently proved some improved sum-product estimates.

We make particular use of the following lemma of Li \cite{Liconvex}.

\begin{lemma}\label{theorem:li}
For any finite $A,B \subseteq \mathbb{R}$ such that $0\notin{A,B}$,
$$E_{1.5}(A)^2|B|^2 \leq E_2(A,AB)E_3(A)^{2/3}E_3(B)^{1/3}.$$
\end{lemma}

Note that the statement of Lemma \ref{theorem:li} presented in \cite{Liconvex} concerns additive rather than multiplicative energy, but the proof is identical after substituting multiplication for addition, taking care to avoid dividing by zero. Indeed, to avoid the problem of division by zero, we will assume throughout the proof of Theorem \ref{theorem:mainreal} that $-1,0,1\notin{A}$.

\subsection{Bounding the multiplicity of particular product sets}

We shall exploit Lemma \ref{theorem:li} for the sets $A$ and $A+1$. To this end we wish to control the multiplicities of several product sets in terms of the size of the set $A(A+1)$. For this we prove the following lemma:

\begin{lemma}\label{theorem:control}
Given finite $A,B \subseteq \mathbb{R}$ and a parameter $1 \leq t \leq |A|,|B|$, write $S_t(A,B)$ for the set of $s \in AB$ such that $\left|A \cap sB^{-1}\right|\geq t$. Then we have
$$\left|S_t(A,B)\right|\ll \frac{|A(A+1)|^2|B|^2}{|A|t^3}  $$
\end{lemma}

Before proving Lemma \ref{theorem:control} we remark that it depends on the Szemer\'edi-Trotter \cite{ST} incidence theorem for points and lines. 

\begin{theorem}[Szemer\'edi-Trotter]\label{theorem:ST}
Let $P$ and $L$ be a set of points and lines respectively in $\mathbb{R}^2$. Then the number $I(P,L)$ of incidences between points in $P$ and lines in $L$ is $O\left(|P|^{2/3}|L|^{2/3}+|P|+|L|\right)$. 
\end{theorem}

This has the following standard consequence, which we shall make use of.

\begin{corollary}
Let $L$ be a set of lines in $\mathbb{R}^2$. Then the number of points incident to at least $k$ lines in $L$ is $O\left(\frac{|L|^2}{k^3}+\frac{|L|}{k}\right)$.
\end{corollary}

\begin{proof}
Let $P$ be the set of points incident to at least $k$ lines in $L$. Then $|P|k\leq I(P,L)$. Comparing to the upper bound in the Szemer\'edi-Trotter theorem shows that $|P|\ll \frac{|L|^2}{k^3}+\frac{|L|}{k}$ as required. 
\end{proof}

We now prove Lemma \ref{theorem:control}. Let $L=\left\{l_{\alpha b}: \alpha \in A(A+1),b \in B\right\}$, where $l_{\alpha b}$ is the line given by $y=(\alpha x -1)b$. For each $t \leq |A|,|B|$ we let $P_t$ be the set of points incident to at least $t$ lines in $L$. By the corollary to Szemer\'edi-Trotter we know that $|P_t|\ll \frac{|L|^2}{t^3}+\frac{|L|}{t}$. Since $t \leq |A|,|B|$ and $|L|=|A(A+1)||B|$ we therefore have

\begin{equation}\label{eq:upper}
|P_t|\ll \frac{|A(A+1)|^2|B|^2}{t^3}
\end{equation}

We now bound $|P_t|$ from below by $\left|S_t(A,B)\right||A|$. Suppose that $(s,a) \in S_t(A,B)\times A$. Since $s \in S_t(A,B)$ there are at least $t$ solutions to the equation $s=ab$ with $a \in A$ and $b \in B$. Index the solution pairs as $(a_i,b_i)$ for $i=1,2,\ldots,t$. For any $a \in A$ and any $1\leq i \leq t$ we have

\begin{align*}
s&=a_ib_i= \left(\frac{\alpha(i,a)}{a}-1\right)b_i
\end{align*}

where $\alpha(i,a)=a(a_i+1)$. This is the same as saying that 

$$\left(\frac{1}{a},s\right)\in l_{\alpha(i,a),b_i}.$$

One can verify that the lines $l_{\alpha(i,a),b_i}$ are all distinct as $i$ ranges from $1$ to $t$. So we deduce that $\left(\frac{1}{a},s\right)\in P_t$ for each $(s,a)\in S_t(A,B)\times A$ and conclude that

\begin{equation}\label{eq:lower}
|P_t|\geq \left|S_t(A,B)\right||A|.
\end{equation} 

Comparing (\ref{eq:lower}) and (\ref{eq:upper}) yields $\left|S_t(A,B)\right|\ll \frac{|A(A+1)|^2|B|^2}{|A|t^3}$
as required. \qed

\subsection{Consequences of bounded multiplicity}

We now use our control of multiplicities, along with Lemma \ref{theorem:li}, to obtain bounds on various forms of multiplicative energy.

\begin{corollary}\label{theorem:2to1.5} $\frac{|A|^{11}}{|A(A+1)|^5}\ll E_{1.5}(A)E_{1.5}(A+1)$.
\end{corollary}

\begin{proof}
By Cauchy-Schwarz we have

\begin{equation}\label{eq:cs}
\frac{|A|^4}{|A(A+1)|}\leq E_2(A,A+1) \leq E_2(A)^{1/2}E_2(A+1)^{1/2}.
\end{equation}

Let $\Delta\geq 1$ be a parameter that will be fixed later. Write $\mu(x)=\left|A\cap x A^{-1}\right|$. Lemma \ref{theorem:control} implies that

\begin{align*}
E_2(A)&=\sum_{\mu(x) \leq \Delta}\mu(x)^2+\sum_{\mu(x) > \Delta}\mu(x)^2\\
&\ll \Delta^{1/2} E_{1.5}(A) + \sum_{j=0}^{\log_2 |A|}\left|S_{2^j\Delta}(A,A)\right|\left({2^j}\Delta\right)^2\\
&\ll \Delta^{1/2} E_{1.5}(A) + \frac{|A(A+1)|^2|A|}{\Delta}.
\end{align*} 

We then set $\Delta=\frac{|A(A+1)|^{4/3}|A|^{2/3}}{E_{1.5}(A)^{2/3}}$ to obtain 

\begin{equation}\label{eq:1}
E_2(A)\ll |A(A+1)|^{2/3}|A|^{1/3}E_{1.5}^{2/3}(A).
\end{equation}

Applying (\ref{eq:1}) to the set $-A-1$ yields

\begin{align}
\nonumber E_2(A+1)&=E_2(-A-1)\\
& \nonumber \ll \left|(-A-1)(-A)\right|^{2/3}|A|^{1/3}E_{1.5}^{2/3}(-A-1)\\
& \label{eq:2} =\left|A(A+1)\right|^{2/3}|A|^{1/3}E_{1.5}^{2/3}(A+1).
\end{align}

Combining (\ref{eq:cs}),(\ref{eq:1}) and (\ref{eq:2}) gives the required result.
\end{proof}

\begin{corollary}\label{theorem:3to0} $E_3(A),E_3(A+1) \lesssim |A(A+1)|^2|A|.$
\begin{proof}
By Lemma \ref{theorem:control} we have 

\begin{align*}
E_3(A)&\leq \sum_{j=0}^{\log_2 |A|}\left|S_{2^j}(A,A)\right|2^{3j}\\
&\ll \sum_{j=0}^{\log_2 |A|} \left|A(A+1)\right|^2|A|\\
&\lesssim |A(A+1)|^2|A|
\end{align*}

which gives the first bound. The proof of the second bound follows by applying this result to the set $-A-1$.
\end{proof}
\end{corollary}

\begin{corollary}\label{theorem:2to0} $$E_2(A,A(A+1))\ll |A(A+1)|^{5/2}$$ and $$E_2(A+1,A(A+1))\ll |A(A+1)|^{5/2}.$$
\end{corollary}
\begin{proof}
Let $\Delta\geq 1$ be a parameter that will be fixed later. Write $\mu(x)=\left|A(A+1)\cap x A^{-1}\right|$ By Lemma \ref{theorem:control} we have

\begin{align*}
E_2(A,A(A+1))&\approx \sum_{\mu(x) \leq \Delta}\mu(x)^2+\sum_{\mu(x)>\Delta}\mu(x)^2\\
&\ll |A||A(A+1)|\Delta + \sum_{j=0}^{\log |A|}\left|S_{2^j\Delta}(A,A(A+1))\right|\left(2^j\Delta\right)^2\\
&\ll |A||A(A+1)|\Delta + \frac{|A(A+1)|^4}{|A|\Delta}.
\end{align*}

We then pick $\Delta=\frac{|A(A+1)|^{3/2}}{|A|}$ to obtain $E_2(A,A(A+1))\ll |A(A+1)|^{5/2}.$ The second bound follows by applying this to the set $-A-1$ and noting that $E_2(-A-1,A(A+1))=E_2(A+1,A(A+1))$.
\end{proof}

\subsection{Completing the proof of Theorem \ref{theorem:mainreal}}
Combining Lemma \ref{theorem:li} and Corollary \ref{theorem:2to1.5} shows that 

\begin{align*}
\frac{|A|^{11}}{|A(A+1)|^5}&\ll E_{1.5}(A)E_{1.5}(A+1)\\
&\ll \frac{E_2(A,A(A+1))^{1/2}E_2(A+1,A(A+1))^{1/2}E_3(A)^{1/2}E_3(A+1)^{1/2}}{|A|^2}.
\end{align*}
 
Applying Corollaries \ref{theorem:3to0} and \ref{theorem:2to0} to bound $E_2(A,A(A+1))$, $E_2(A,A(A+1))$, $E_3(A)$ and $E_3(A+1)$ then yields
$|A|^{24}\ll |A(A+1)|^{19}$ as required. \qed

\section*{Acknowledgements}
The authors are grateful to Misha Rudnev for encouragement and helpful conversations.

\bibliographystyle{plain}
\bibliography{expandersbib}

\begin{thebibliography}{10}

\bibitem{bourgain}
J.~Bourgain.
\newblock More on the sum-product phenomenon in prime fields and its
  applications.
\newblock {\em Int. J. Number Theory}, 1(1):1--32, 2005.

\bibitem{BKT}
J.~Bourgain, N.~Katz, and T.~Tao.
\newblock A sum-product estimate in finite fields and applications.
\newblock {\em Geom. Func. Anal.}, 14(1):27--57, 2004.

\bibitem{elekes}
G.~Elekes.
\newblock On the number of sums and products.
\newblock {\em Acta. Arith}, 81(4):365--367, 1997.

\bibitem{ENR}
G.~Elekes, M.~Nathanson, and I.~Ruzsa.
\newblock Convexity and sumsets.
\newblock {\em Journal of Number Theory}, 83:194--201, 1999.

\bibitem{erdosszemeredi}
P.~Erd\H{o}s and E.~Szemer\'edi.
\newblock On sums and products of integers.
\newblock {\em Stud. Pure Math.}, pages 213--218, 1983.

\bibitem{ford}
K.~Ford.
\newblock Sums and products from a finite set of real numbers.
\newblock {\em Ramanujan J.}, (2):59--66, 1998.

\bibitem{garaev}
M.Z. Garaev.
\newblock An explicit sum-product estimate in $\mathbb{F}_p$.
\newblock {\em Int. Math. Res. Not.}, (11):1--11, 2007.

\bibitem{GS}
M.Z. Garaev and C.Y. Shen.
\newblock On the size of the set ${A(A+1)}$.
\newblock {\em Math. Z.}, 265(1):125--132, 2010.

\bibitem{BG}
J.Bourgain and M.Garaev.
\newblock On a variant of sum-product estimates and explicit exponential sum
  bounds in prime fields.
\newblock {\em Math. Proc. Cambridge Philos. Soc.}, 146(1):1--21, 2009.

\bibitem{KSgeneral}
N.~Katz and C.-Y. Shen.
\newblock Garaev's inequality in finite fields not of prime order.
\newblock {\em Online J. Anal. Comb.}, (3), 2008.

\bibitem{ks}
N.~Katz and C.-Y. Shen.
\newblock A slight improvement to {G}araev$'$s sum product estimate.
\newblock {\em Proc. Amer. Math. Soc.}, 136(7):2499--2504, 2008.

\bibitem{Liconvex}
L.~Li.
\newblock On a theorem of {S}choen and {S}hkredov on sumsets of convex sets.
\newblock Preprint arXiv:1108.4382v1., 2011.

\bibitem{li}
L.~Li.
\newblock Slightly improved sum-product estimates in fields of prime order.
\newblock {\em Acta Arith.}, 147(2):153--160, 2011.

\bibitem{LiORN2}
L.~Li. and O.~Roche-Newton.
\newblock Convexity and a sum-product type estimate.
\newblock Preprint arXiv:1111.5159, to appear Acta Arith.

\bibitem{LiORN1}
L.~Li and O.~Roche-Newton.
\newblock An improved sum-product estimate for general finite fields.
\newblock {\em SIAM Journal of Discrete Mathematics}, 25(3):1285--1296, 2011.

\bibitem{nathanson}
M.B. Nathanson.
\newblock On sums and products of integers.
\newblock {\em Proc. Amer. Math. Soc.}, (125):9--16, 1997.

\bibitem{rudnev}
M.~Rudnev.
\newblock An improved sum-product inequality in fields of prime order.
\newblock {\em Int. Math. Res. Notices}, 2011.

\bibitem{mishaSP2}
M.~Rudnev.
\newblock On new sum-product type estimates.
\newblock Preprint arXiv:1111.4977., 2012.

\bibitem{SS}
T.~Schoen and I.~Shkredov.
\newblock On sumsets of convex sets.
\newblock {\em Combinatorics , Probability and Computing}, 20:793--798, 2011.

\bibitem{shen}
C.-Y. Shen.
\newblock Quantitative sum product estimates on different sets.
\newblock {\em Electron. J. Combin.}, 15(1):7 pp, 2008.

\bibitem{solymosi2}
J.~Solymosi.
\newblock On the number of sums and products.
\newblock {\em Bull. London Math. Soc.}, 37(4):491--494, 2005.

\bibitem{solymosi1}
J.~Solymosi.
\newblock Bounding multiplicative energy by the sumset.
\newblock {\em Adv. Math.}, 222(2):402–408, 2009.

\bibitem{ST}
E.~Szemer\'edi and W.~T.~Trotter Jr.
\newblock Extremal problems in discrete geometry.
\newblock {\em Combinatorica}, 3(3-4):381--392, 1983.

\bibitem{TV}
T.~Tao and V.~Vu.
\newblock {\em Additive Combinatorics}.
\newblock Cambridge University Press, 2006.

\end{thebibliography}

\end{document}